\documentclass{amsart}
\usepackage[
    top=1in,
    bottom=1in,
    left=1in,
    right=1in
]{geometry}
\usepackage[all]{xy}
\usepackage{verbatim}
\usepackage{color}
\usepackage{amsthm}
\usepackage{amssymb}
\usepackage[colorlinks=true]{hyperref}
\usepackage{footmisc}
\usepackage{mathtools}
\usepackage{tabularx}

\numberwithin{equation}{section}

\newtheorem{theorem}[equation]{Theorem}
\newtheorem*{theorem*}{Theorem}

\newtheorem*{conjecture*}{Mamma Conjecture}
\newtheorem*{conjecture1*}{Mamma Conjecture (revisited)}
\newtheorem{proposition}[equation]{Proposition}
\newtheorem{corollary}[equation]{Corollary}
\newtheorem*{corollary*}{Corollary}

\theoremstyle{remark}
\newtheorem{definition}[equation]{Definition}

\newtheorem{example}[equation]{Example}

\theoremstyle{remark}
\newtheorem{remark}[equation]{Remark}

\newcommand{\cA}{{\mathcal A}}
\newcommand{\cB}{{\mathcal B}}

\newcommand{\bbZ}{\mathbb{Z}}

\DeclareMathOperator{\NChow}{NChow} 

\newcommand{\dgcat}{\mathrm{dgcat}} 

\newcommand{\perf}{\mathrm{perf}}

\newcommand{\dg}{\mathrm{dg}}

\newcommand{\Hom}{\mathrm{Hom}}

\newcommand{\op}{\mathrm{op}}

\newcommand{\too}{\longrightarrow}

\let\oldmarginpar\marginpar
\def\marginpar#1{\oldmarginpar{\tiny #1}}

\begin{document}

\title[Independence of the Grothendieck classes of twisted symplectic flag varieties]{Independence of the Grothendieck classes of \\twisted symplectic flag varieties}
\author{Gon{\c c}alo~Tabuada}
\address{Gon{\c c}alo Tabuada, Mathematics Institute, Zeeman Building, University of Warwick, Coventry CV4 7AL UK.}
\email{goncalo.tabuada@warwick.ac.uk}
\date{\today}

\abstract{In this short note, making use of the theory of noncommutative motives, we prove that the Grothendieck classes of the twisted symplectic flag varieties are linearly independent and moreover homogeneous-quadratic independent in the Grothendieck ring of algebraic varieties.}
}

\maketitle

\section{Introduction}\label{sec:intro}
Let $k$ be a field and $\mathrm{Var}(k)$ the category of {\em algebraic varieties}, i.e., reduced separated $k$-schemes of finite type. The {\em Grothendieck ring of algebraic varieties $K_0\mathrm{Var}(k)$}, introduced in a letter from Grothendieck to Serre (consult \cite[letter of 16/08/1964]{SG}), is defined as the quotient of the free abelian group on the set of isomorphism classes of algebraic varieties $[X]$ by the ``scissor'' relations $[X]=[Z]+[X\backslash Z]$, where $Z$ is a closed subvariety of $X$; the multiplication law is induced by the product. Despite the efforts of several mathematicians (consult, for example, the articles \cite{Bittner,Borisov, Kollar, Lunts, Sebag, Naumann}), the structure of the ring $K_0\mathrm{Var}(k)$ still remains nowadays poorly understood. In this short note, we improve our understanding of the Grothendieck ring of algebraic varieties by studying the Grothendieck classes of twisted symplectic flag varieties.
\subsection*{Twisted symplectic flag varieties}
Let $k$ be a field of characteristic $\neq 2$. Given a central simple $k$-algebra $A$, let us write $\mathrm{deg}(A):=\sqrt{\mathrm{dim}_k(A)}$ for its degree, $\mathrm{per}(A)$ for its period and $[A] \in \mathrm{Br}(k)$ for its Brauer class.

Given a central simple $k$-algebra with involution of symplectic type\footnote{If a central simple $k$-algebra $A$ admits an involution of symplectic type, then $\mathrm{deg}(A)$ is even and $\mathrm{per}(A)\leq 2$.} $(A,\ast)$ and a sequence of integers $1 \leq n_1 < \ldots < n_r \leq \mathrm{deg}(A)/2$, let us write $\mathrm{Flag}(A,\ast; n_1, \ldots, n_r)$ for the associated twisted symplectic flag variety; consult \cite[\S5]{Merkurjev}. Recall from {\em loc. cit.} that this algebraic variety has dimension 
$$\mathrm{deg}(A)^2/4 - \sum_{j=1}^r \binom{n_j - n_{j-1}}{2} - (\mathrm{deg}(A)/2 - n_r)^2$$
and that its $R$-points, with $R$ a commutative $k$-algebra, are given by
$$ \mathrm{Flag}(A,\ast; n_1, \ldots, n_r)(R)= \{(I_1, \ldots, I_r)\,|\, I_1 \subseteq  \ldots \subseteq I_k \subseteq A \otimes_kR\}\,,$$
where each $I_j$ is a right ideal of $A\otimes_k R$ and a direct summand $R$-submodule of $A\otimes_kR$ of rank $n_j \mathrm{deg}(A)$ such that $(\ast \otimes_k R)(I_j)\cdot I_j =\{0\}$. 
\begin{example}[Twisted symplectic Grassmannians]\label{ex:1}
In the particular case where $r=1$, a twisted symplectic flag variety reduces to a twisted symplectic Grassmannian $\mathrm{Gr}(A,\ast, n)$, with $1\leq n \leq \mathrm{deg}(A)/2$. Moreover, in the particular case where $n=1$, resp. $n=\mathrm{deg}(A)/2$, a twisted symplectic Grassmannian reduces furthermore to a Severi-Brauer variety $\mathrm{SB}(A)$, resp. to a twisted Lagrangian variety $\mathrm{LGr}(A,\ast)$.
\end{example}
\begin{example}[Conics]\label{ex:conics}
Given a quaternion $k$-algebra $Q=(a,b)$, recall that its canonical involution $q \mapsto \overline{q}$ is of symplectic type and that the associated Severi-Brauer variety $\mathrm{SB}(Q)$ is the following conic:
$$C(a,b):=(-ax^2 - by^2 + abu^2=0) \subset \mathbb{P}^2\,.$$ 
Therefore, since the assignment $(a,b) \mapsto C(a,b)$ induces a one-to-one correspondence between quaternion $k$-algebras up to isomorphism and conics up to isomorphism, we hence conclude from the above Example \ref{ex:1} that every conic is a twisted symplectic flag variety. 
\end{example}
\begin{example}[Symplectic flag varieties]\label{ex:split}
In the particular case where the central simple $k$-algebra $A$ is {\em $k$-split}, i.e., when $A=\mathrm{End}_k(V)$ with $V$ a $k$-vector space of even dimension $2d$, there is a unique involution of symplectic type $\ast$ on $A$ up to isomorphism. Specifically, it is the adjoint involution $\ast_{\omega}$ of the unique non-degenerate symplectic form $\omega\colon V \times V \to k$ on $V$ up to similarity. In this particular case, a twisted symplectic flag variety reduces to a classical symplectic flag variety $\mathrm{Flag}(2d; n_1, \ldots, n_r)$ with $1\leq n_1 < \ldots < n_r \leq d$. Recall that the $k$-rational points of the classical symplectic flag variety are given by 
$$ \mathrm{Flag}(2d; n_1, \ldots, n_r)(k) =\{(V_1, \ldots, V_r) \,|\, V_1 \subseteq \ldots \subseteq V_r \subseteq V\}\,,$$
where each $V_j$ is a $k$-subspace of $V$ of dimension $n_j$ such that $\omega(V_j,V_j)=\{0\}$. 
\end{example}
\begin{remark}[Forms]\label{rk:forms}
Following Example \ref{ex:split}, every twisted symplectic flag variety $\mathrm{Flag}(A,\ast; n_1, \ldots, n_r)$ becomes isomorphic to the classical symplectic flag variety after extension of scalars to a splitting field of $A$. Conversely, as explained in \cite[\S5]{Merkurjev}, those algebraic varieties that become isomorphic to the classical symplectic flag variety after an extension of scalars are precisely the twisted symplectic flag varieties. Hence, the twisted symplectic flag varieties are the $k$-forms\footnote{Recall that an algebraic variety $X$, defined over $k$, is a {\em $k$-form} of an algebraic variety $Y$, defined over $\overline{k}$, if $X\times_k \overline{k} \simeq Y$.} of the classical symplectic flag varieties.
\end{remark}
\begin{example}[Quaternion algebras]\label{ex:quaternions}
In the particular case where $A=\mathrm{End}_Q(W)$, with $Q$ a quaternion division $k$-algebra and $W$ a right $Q$-vector space, there is a one-to-one correspondence $h \mapsto \ast_h$ between non-degenerate Hermitian forms\footnote{Hermitian forms $h\colon W \times W \to Q$ with respect to the canonical involution $q \mapsto \overline{q}$ of $Q$.} $h\colon W \times W \to Q$ on $W$ up to similarity and involutions of symplectic type $\ast_h$ on $A$ up to isomorphism. Hence, in this particular case, the $k$-rational points may be described as follows 
$$ \mathrm{Flag}(\mathrm{End}_Q(W), \ast_h; n_1, \ldots, n_r)(k) =\{(W_1, \ldots, W_r) \,|\, W_1 \subseteq \ldots \subseteq W_r \subseteq W\}\,,$$
where each $W_j$ is a right $Q$-subspace of $W$ of dimension $n_j/2$ such that $h(W_j, W_j)=\{0\}$.
\end{example}
\subsection*{Statement of results}
Let $k$ be a field of characteristic zero.
\begin{definition}
Let $R$ be a commutative ring. A family $\{a_i\}_{i \in I}$ of elements of $R$ is called {\em homogeneous-quadratic independent} if for every finite subset of indices $\{i_1, \ldots, i_m\} \subseteq I$, the only homogeneous polynomial $f \in \mathbb{Z}[x_1, \ldots, x_m]$ of degree $2$ that satisfies $f(a_{i_1}, \ldots, a_{i_m})=0$ is the zero polynomial.
\end{definition}
\begin{theorem}\label{thm:main}
Let $\{(A_i,\ast_i)\}_{i \in I}$ be a family of central simple $k$-algebras with involution of symplectic type and $\{1\leq n^i_1 < \ldots < n^i_{r_i}\leq \mathrm{deg}(A_i)/2\}_{i \in I}$ a family of sequences of integers. Consider the associated family of Grothendieck classes of twisted symplectic flag varieties in the Grothendieck ring of algebraic varieties:
\begin{equation}\label{eq:Grothendieck}
\{[\mathrm{Flag}(A_i, \ast_i; n^i_1, \ldots, n^i_{r_i})]\}_{i \in I}\,.
\end{equation}
\begin{itemize}
\item[a)] If the Brauer classes $[A_i]$, with $i \in I$, are pairwise distinct, then the family \eqref{eq:Grothendieck} is linearly independent.
\item[b)] If the Brauer classes $[A_i], [A_i \otimes A_j]$, with $i, j \in I$ and $i\neq j$, are pairwise distinct, then the family \eqref{eq:Grothendieck} is homogeneous-quadratic independent.
\end{itemize}
\end{theorem}
\begin{remark}[Assumptions]
Item b) of Theorem \ref{thm:main} is false without the assumption that the Brauer classes $[A_i], [A_i \otimes A_j]$, with $i, j \in I$ and $i\neq j$, are pairwise distinct. For example, in the particular case of conics (see Example \ref{ex:conics}), given a quaternion division $k$-algebra $Q=(a,b)$, Koll\'ar\footnote{Koll\'ar worked over a number field or over the function field of an algebraic surface over $\mathbb{C}$. In subsequent work, Hogadi \cite{Hogadi} removed these restrictions on the base field $k$.} proved in \cite[Thm.~2]{Kollar} that the following homogeneous quadratic relation holds in the Grothendieck ring of algebraic varieties:
$$[C(a,b)] [\mathbb{P}^1] - [C(a,b)] [C(a,b)]=0$$ 
 Note that this does not contradicts item b) of Theorem \ref{thm:main} because in this particular case we have $C(1,1)=\mathbb{P}^1$ and the Brauer classes $[Q]$ and $[Q\otimes (1,1)]$ are the same.
\end{remark}
\begin{remark}[Number fields]\label{rk:numbers}
In the case where $k$ is a number field, it is well-known that there is a one-to-one correspondence $Q \mapsto \mathrm{Ram}(Q)$ between quaternion $k$-algebras up to isomorphism and finite subsets of even cardinality of the set of places of $k$ (a.k.a. ramification subsets). Moreover, every central simple $k$-algebra $A$ of even degree and $\mathrm{per}(A)\leq 2$ can be written uniquely as a matrix algebra over a quaternion algebra $M_{n \times n}(Q)$. This implies that $\mathrm{Ram}(A)=\mathrm{Ram}(Q)$. Consequently, the assumption of Theorem \ref{thm:main} a) can be re-written as follows: the sets $\mathrm{Ram}(A_i)$, with $i \in I$, are pairwise distinct. In the same vein, since $\mathrm{Ram}(A_i \otimes A_j) = \big(\mathrm{Ram}(A_i) \cup \mathrm{Ram}(A_j)\big) \backslash \big(\mathrm{Ram}(A_i) \cap \mathrm{Ram}(A_j)\big)$, the assumption of Theorem \ref{thm:main} b) can be re-written as: the sets $\mathrm{Ram}(A_i), \big(\mathrm{Ram}(A_i) \cup \mathrm{Ram}(A_j)\big) \backslash \big(\mathrm{Ram}(A_i) \cap \mathrm{Ram}(A_j)\big)$, with $i, j \in I$ and $i\neq j$, are pairwise distinct.
\end{remark}
\begin{example}[Infinite family over $\mathbb{Q}$]
In the particular case where $k=\mathbb{Q}$, recall that its set of places is given by $\{p\,|\,p\,\, \text{prime}\,\,\text{number}\} \cup \{\infty\}$. Consider the infinite family of quaternion $\mathbb{Q}$-algebras $(-1,-p)$, with $p \equiv 3\, (\text{mod}\, 4)$, where $p$ is a prime number. Since $\mathrm{Ram}(-1,-p)=\{p, \infty\}$, the sets $\mathrm{Ram}(-1,-p)=\{p,\infty\}$, with $p \equiv 3\, (\text{mod}\, 4)$, are pairwise distinct. Similarly, the sets 
\begin{eqnarray*}
\mathrm{Ram}(-1,-p) = \{p, \infty\} && \big(\mathrm{Ram}(-1,-p)\cup \mathrm{Ram}(-1,-q) \big)\backslash \big(\mathrm{Ram}(-1,-p) \cap \mathrm{Ram}(-1,-q)\big)=\{p,q\}\,,
\end{eqnarray*}
with $p, q  \equiv 3\, (\text{mod}\, 4)$ and $p \neq q$, are pairwise distinct. Therefore, given any family of central simple $\mathbb{Q}$-algebras with involution of symplectic type $\{(\mathrm{End}_{(-1,-p)}(W_p), \ast_{h_p})\}_{p \equiv 3\, (\text{mod}\, 4)}$ (consult Example \ref{ex:quaternions}) and any family of sequences of integers $\{1 \leq n_1^p < \cdots < n^p_{r_p}\leq \mathrm{dim}_{(-1,-p)}(W_p)\}_{p \equiv 3\, (\text{mod}\, 4)}$, it follows from Remark \ref{rk:numbers} that the associated infinity family of Grothendieck classes of twisted symplectic flag varieties 
$$ \{[\mathrm{Flag}(\mathrm{End}_{(-1,-p)}(W_p), \ast_{h_p}; n_1^p, \ldots, n^p_{r_p})]\}_{p \equiv 3\, (\text{mod}\, 4)}$$
is linearly independent and homogeneous-quadratic independent. 
\end{example}
Recall from \cite[\S II]{Lam} the definition of the Witt ring $W(k)$ of $k$ and of the fundamental ideal $I(k) \subset W(k)$. In what follows, we will denote by $I(k)^3$ the $3$rd power of $I(k)$.
\begin{corollary}[Forms]\label{cor:forms}
Assume that $I(k)^3=0$. Given a sequence of integers $1 \leq n_1 < \ldots < n_r \leq d$, the following family of Grothendieck classes in the Grothendieck group of algebraic varieties
\begin{equation}\label{eq:family}
\{[X]\,|\, X\,\,\text{is}\,\,\text{a}\,\,k\text{-}\text{form}\,\,\text{of}\,\,\mathrm{Flag}(2d;n_1, \ldots, n_r)\}
\end{equation}
is linearly independent.
\end{corollary}
\begin{remark}
The condition $I(k)^3=0$ holds, for example, when $k$ is a $p$-adic field, a non-formally real number field or a field of transcendence degree $2$ over an algebraically closed field.
\end{remark}
\begin{proof}
Let $X$ be a $k$-form of the classical symplectic flag variety $\mathrm{Flag}(2d;n_1, \ldots, n_r)$. Following Remark \ref{rk:forms}, there exists a unique central simple $k$-algebra with involution of symplectic type $(A,\ast)$, with $\mathrm{deg}(A)=2d$, such that $X\simeq \mathrm{Flag}(A, \ast;n_1, \ldots, n_r)$. Moreover, as proved in \cite[Thm.~A]{LewisT}, when $I(k)^3=0$, every central simple $k$-algebra $A$ of even degree and $\mathrm{per}(A)\leq 2$, admits a unique involution of symplectic type up to isomorphism. Therefore, the proof follows now from item a) of Theorem \ref{thm:main}.  
\end{proof}
Corollary \ref{cor:forms} shows that, whenever $I(k)^3=0$, the Grothendieck classes of the forms of the classical symplectic flag variety are all linearly independent. In particular, they are all pairwise distinct. 
\begin{remark}[Non-formally real number fields]
As explained in Remark \ref{rk:numbers}, when $k$ is a (non-formally real) number field, the above family of Grothendieck classes \eqref{eq:family} is always infinite.
\end{remark}
\begin{remark}[Fields of transcendence degree $2$ over an algebraically closed field]
When $k$ is a field of transcendence degree $2$ over an algebraically closed field, every central simple $k$-algebra $A$ of even degree and $\mathrm{per}(A)\leq 2$ can be written uniquely as a matrix algebra over a quaternion $k$-algebra. Moreover, the quaternion $k$-algebras can be classified geometrically in terms of their ramification curves inside a surface $S$ which models $k$; consult \cite[\S3]{book-Brauer}. In particular, the family of Grothendieck classes \eqref{eq:family} is always infinite.
\end{remark}
\section{Preliminaries}\label{sec:preliminaries}
Let $k$ be a base field of characteristic zero.
\subsection*{Dg categories}\label{sub:dg}
 Let us write $\dgcat(k)$ for the category of (small) dg categories and $\dgcat_{\mathrm{sp}}(k)$ for the full subcategory of smooth proper dg categories in the sense of Kontsevich; consutl \cite{ICM-Keller,Miami,finMot,IAS}. Examples of smooth proper dg categories include, for example, the finite-dimensional $k$-algebras of finite global dimension $A$ as well as the dg categories of perfect complexes $\perf_\dg(X)$ associated to smooth proper $k$-schemes $X$. As explained in \cite[\S1.7]{book}, the symmetric monoidal category $(\dgcat_{\mathrm{sp}}(k), \otimes)$ is rigid\footnote{Recall that a symmetric monoidal category is called {\em rigid} if all its objects are dualizable.}, with the dual of a smooth proper dg category $\cA$ being its opposite dg category $\cA^\op$.
\subsection*{Noncommutative Chow motives}\label{sub:NCmotives}
Recall from \cite[\S4.1]{book} the definition of the category of {\em noncommutative Chow motives} $\NChow(k)$. This category is additive, rigid symmetric monoidal, idempotent complete, and comes equipped with a symmetric monoidal functor $U(-) \colon \dgcat(k)_{\mathrm{sp}} \to \NChow(k)$. Given smooth proper dg categories $\cA$ and $\cB$, we have an isomorphism 
$$\Hom_{\NChow(k)}(U(\cA), U(\cB))\simeq K_0(\cA^\op \otimes \cB)\,,$$
where the right-hand side stands for the Grothendieck group of $\cA^\op \otimes \cB$. Moreover, the composition law on $\NChow(k)$ is induced by the (derived) tensor product of bimodules, and the identity of $U(\cA)$ is the Grothendieck class of the diagonal $\cA\text{-}\cA$-bimodule $\cA$.
\section{Proof of Theorem \ref{thm:main}}\label{sec:proof}
Let us write $K_0(\NChow(k))$ for the Grothendieck ring of the additive symmetric monoidal category of noncommutative Chow motives. Consider the full subcategory ${}_2\mathrm{CSA}(k):=\{U(A)\,|\, \mathrm{per}(A)\leq 2\}$ of $\NChow(k)$ and its closure under finite direct sums ${}_2\mathrm{CSA}(k)^\oplus \subset \NChow(k)$. Note that ${}_2\mathrm{CSA}(k)^\oplus$ is an additive symmetric monoidal full subcategory of $\NChow(k)$. In what follows, we will write $K_0({}_2\mathrm{CSA}(k)^\oplus)$ for its Grothendieck ring.
\begin{proposition}\label{prop:aux}
\begin{itemize}
\item[(i)] The category ${}_2\mathrm{CSA}(k)^\oplus$ is idempotent complete.
\item[(ii)] The inclusion ${}_2\mathrm{CSA}(k)^\oplus \subset \NChow(k)$ induces an inclusion of ring $K_0({}_2\mathrm{CSA}(k)^\oplus) \subset K_0(\NChow(k))$.
\item[(iii)] We have a ring isomorphism:
\begin{eqnarray}\label{eq:ring}
K_0({}_2\mathrm{CSA}(k)^{\oplus}) \stackrel{\simeq}{\too} \bbZ[{}_2\mathrm{Br}(k)] && U(A_1) \oplus \cdots \oplus U(A_m) \mapsto [A_1] + \cdots + [A_m]\,,
\end{eqnarray}
where $\bbZ[{}_2\mathrm{Br}(k)]$ stands for the group ring of the $2$-torsion subgroup ${}_2\mathrm{Br}(k)$ of $\mathrm{Br}(k)$.
\end{itemize}
\end{proposition}
\begin{proof}
\begin{itemize}
\item[(i)] Let $U(A_1) \oplus \cdots \oplus U(A_m)$ be an object of the category ${}_2\mathrm{CSA}(k)^\oplus$. Consider the subcategory 
$$ {}_2\mathrm{CSA}(A_1, \ldots, A_m)^\otimes := \{U(\otimes_{j \in J} A_j)\,|\, J \subseteq \{1, \ldots, m\} \} \subset {}_2\mathrm{CSA}(k)$$ as well as its closure under finite direct sums ${}_2\mathrm{CSA}(A_1, \ldots, A_m)^{\otimes, \oplus} \subset {}_2\mathrm{CSA}(k)^\oplus$. Note that since $\mathrm{per}(A_j)\leq 2$, with $1\leq j \leq m$, the Brauer classes $\{[\otimes_{j \in J} A_j]\}_{J \subseteq \{1, \ldots, m\}}$ form a subgroup of ${}_2\mathrm{Br}(k)$. Thanks to \cite[Prop.~2.28]{Separable}, this implies that the category ${}_2\mathrm{CSA}(A_1, \ldots, A_m)^{\otimes, \oplus}$ is idempotent complete. Consequently, since the object $U(A_1) \oplus \cdots \oplus U(A_m)$ belongs to the category ${}_2\mathrm{CSA}(A_1, \ldots, A_m)^{\otimes, \oplus}$, every idempotent endomorphism of this object will split in ${}_2\mathrm{CSA}(A_1, \ldots, A_m)^{\otimes, \oplus}$ and, hence, in the category ${}_2\mathrm{CSA}(k)^\oplus$.
\item[(ii)] Let $A_1, \ldots, A_m$ and $B_1, \ldots, B_n$ be two families of central simple $k$-algebras. Given a noncommutative Chow motive $N\!M \in \NChow(k)$, recall from \cite[Prop.~4.9]{Tits} that we have the following implication:
\begin{eqnarray}\label{eq:cancellation}
N\!\!M \oplus \oplus_{i=1}^m U(A_i) \simeq N\!\!M \oplus \oplus_{j=1}^n U(B_j) & \Rightarrow & m=n\,\,\,\mathrm{and}\,\, \oplus_{i=1}^m U(A_i) \simeq \oplus_{j=1}^n U(B_j)\,.
\end{eqnarray}
Now, recall that the group completion of an arbitrary monoid $(M,+)$ is defined as the quotient of the product $M \times M$ by the following equivalence relation:
\begin{equation}\label{eq:cancellation-1}
(m,n) \simeq (m',n') := \exists \,r \in M\,\,\mathrm{such}\,\,\mathrm{that}\,\, m+n'+r = n + m' + r\,.
\end{equation}
Let us write $K_0(\NChow(k))^+$ for the semi-ring of the additive symmetric monoidal category $\NChow(k)$. Concretely, $K_0(\NChow(k))^+$ is the set of isomorphism classes of noncommutative Chow motives equipped with the addition, resp. multiplication, law is induced by $\oplus$, resp. $\otimes$. In the same vein, let us write $K_0({}_2\mathrm{CSA}(k)^\oplus)^+$ for the semi-ring of the additive symmetric monoidal category ${}_2\mathrm{CSA}(k)^\oplus$. Clearly, the inclusion $\mathrm{CSA}(k)^\oplus \subset \NChow(k)$ gives rise to an inclusion $K_0(\mathrm{CSA}(k)^\oplus)^+ \subset K_0(\NChow(k))^+$ of semi-rings. Therefore, by combining the above definition of group completion \eqref{eq:cancellation-1} with the above implication \eqref{eq:cancellation}, we obtain an inclusion of rings $K_0({}_2\mathrm{CSA}(k)^\oplus) \subset K_0(\NChow(k))$.

\item[(iii)] Recall that given a central simple $k$-algebra $A$, there exists a unique central division $k$-algebra $D$ such that $[A]=[D]$ in the Brauer group $\mathrm{Br}(k)$. Making use of this well-known fact, note that the following ring isomorphism
\begin{eqnarray*}
 \bbZ[{}_2\mathrm{Br}(k)] \stackrel{\simeq}{\too} K_0({}_2\mathrm{CSA}(k)^{\oplus})  && \alpha_1[A_1] + \cdots + \alpha_m[A_m] \mapsto \alpha_1[U(D_1)] \oplus \cdots \oplus \alpha_m[U(D_m)]
\end{eqnarray*}
is the inverse of the above ring homomorphism \eqref{eq:ring}.
\end{itemize}
\end{proof}
As proved in \cite[Prop.~4.1]{Tits}, the assignment $X \mapsto U(\perf_\dg(X))$, with $X$ a smooth projective $k$-scheme, gives rise to a motivic measure $\mu_{\mathrm{nc}}\colon K_0\mathrm{Var}(k) \to K_0(\NChow(k))$.  

Let $(A,\ast)$ be a central simple $k$-algebra with involution of symplectic type and $1\leq n_1 < \ldots < n_r \leq \mathrm{deg}(A)/2$ a sequence of integers. By combining \cite[Thm.~2.1]{Homogeneous} with \cite[\S5.1]{Panin} (in the case where the algebraic group $G$ is the symplectic group and the parabolic subgroup $P$ is the one determined by the nodes $n_1, \ldots, n_r$ of the Dynkin diagram), we obtain the computation in the category of noncommutative Chow motives
$$ U(\perf_\dg(\mathrm{Flag}(A,\ast; n_1, \ldots, n_r))) \simeq U(k)^{\oplus \alpha} \oplus U(A)^{\oplus \beta}\,,$$
where $\alpha$ and $\beta$ are two positive integers that depend on $\mathrm{deg}(A)$ and on $n_1, \ldots, n_r$. Consequently, making use of Proposition \ref{prop:aux}, we conclude that
\begin{equation}\label{eq:measure}
\mu_{\mathrm{nc}}([\mathrm{Flag}(A,\ast; n_1, \ldots, n_r)]) = \alpha [k] + \beta [A] \in \bbZ[{}_2\mathrm{Br}(k)]\,.
\end{equation} 
\subsection*{Proof of item a)}
Let $\{i_1, \ldots, i_m\} \subset I$ be an arbitrary, but fixed, finite subset of indices. In order to prove item a), we need to show the following implication (with $\lambda_j\in \mathbb{Z}$):
\begin{eqnarray}\label{eq:implication}
\sum_{j \in \{i_1, \ldots, i_m\}} \lambda_j [\mathrm{Flag}(A_j, \ast_j; n_1^j, \ldots, n^j_{r_j})]=0 & \Rightarrow & \lambda_{j} =0 \,\,\,\,\,\,\forall_j\,.
\end{eqnarray}
Making use of the motivic measure $\mu_{\mathrm{nc}}\colon K_0\mathrm{Var}(k) \to K_0(\NChow(k))$ and of the above computation \eqref{eq:measure}, the left-hand side of \eqref{eq:implication} implies that 
\begin{equation}\label{eq:computation}
\sum_{j \in \{i_1, \ldots, i_m\}} \lambda_j (\alpha_j [k] + \beta_j[A_j])=0
\end{equation}
in the group ring $\bbZ[{}_2\mathrm{Br}(k)]$. 

Now, recall that, by assumption, the Brauer classes $[A_j]$, with $j \in \{i_1, \ldots, i_m\}$, are pairwise distinct. We consider first the case when $[A_j]\neq [k]$ for every $j \in \{i_1, \ldots, i_m\}$. In this case, by evaluating the above equality \eqref{eq:computation} at the Brauer classes $[A_j]$, with $j \in \{i_1, \ldots, i_m\}$, we conclude that $\lambda_j \beta_j=0$ for every $j \in \{i_1, \ldots, i_m\}$. Since $\beta_j >0$, this hence implies that $\lambda_j=0$ for every $j \in \{i_1, \ldots, i_m\}$. Let us now consider the case where $[A_l]=[k]$ for some $l \in \{i_1, \ldots, i_m\}$. In this case, by evaluating the above equality \eqref{eq:computation} at the Brauer classes $[A_j]$, with $j \in \{i_1, \ldots, i_m\}$, we conclude that $\lambda_j \beta_j =0$ for every $j \in \{i_1, \ldots, i_m\}\backslash \{l\}$ and that $\sum_{j \in \{i_1, \ldots, i_m\}} \lambda_j \alpha_j + \lambda_l \beta_l =0$. Since $\beta_j>0$, this implies that $\lambda_j=0$ for every $j \in \{i_1, \ldots, i_m\}\backslash \{l\}$ and that $\lambda_l \alpha_l + \lambda_l \beta_l =0$. Moreover, since $\alpha_l, \beta_l >0$, this implies furthermore that $\lambda_l=0$. This concludes the proof of item a).
\subsection*{Proof of item b)}
Similarly to item a), let $\{i_1, \ldots, i_m\} \subset I$ be an arbitrary, but fixed, finite subset of indices. We need to show the following implication (with $\lambda_{j, l}\in \mathbb{Z}$):
\begin{eqnarray}\label{eq:implication-quadratic}
\sum_{j, l \in \{i_1, \ldots, i_n\}\,\text{with}\,\mathrm{sub}(j) \leq \mathrm{sub}(l)} \lambda_{j,l} [\mathrm{Flag}(A_j, \ast_j; n_1^j, \ldots, n^j_{r_j})] [\mathrm{Flag}(A_l, \ast_l; n_1^l, \ldots, n^l_{r_l})]=0 & \Rightarrow & \lambda_{j,l} =0\,\,\,\,\,\,\forall_{j, l}\,.
\end{eqnarray}
where $\mathrm{sub}(j)$, resp. $\mathrm{sub}(l)$, stands for the subscript of $j$, resp. of $l$. Making use of the motivic measure $\mu_{\mathrm{nc}}\colon K_0\mathrm{Var}(k) \to K_0(\NChow(k))$ and of the above computation \eqref{eq:measure}, the left-hand side of \eqref{eq:implication-quadratic} implies~that 
\begin{equation}\label{eq:computation-1}
\sum_{j, l \in \{i_1, \ldots, i_m\}\,\text{with}\,\mathrm{sub}(j) \leq \mathrm{sub}(l)} \lambda_{j,l} (\alpha_j [k] + \beta_j[A_j]) (\alpha_l [k] + \beta_l[A_l])=0
\end{equation}
in the group ring $\bbZ[{}_2\mathrm{Br}(k)]$. Note also that we have the following computation:
\begin{eqnarray}
(\alpha_j [k] + \beta_j[A_j]) (\alpha_l [k] + \beta_l[A_l]) & = & \begin{cases} (\alpha_j^2 + \beta^2_j)[k] + 2\alpha_j \beta_j [A_j]& j=l \\ \alpha_j\alpha_l [k] + \beta_j \alpha_l [A_j] + \alpha_j \beta_l [A_l] + \beta_j \beta_l[A_j \otimes A_l] & j\neq l\,.  \end{cases}
\end{eqnarray}
Now, recall that, by assumption, the Brauer classes $[A_j], [A_j \otimes A_l]$, with $j, l \in I$ and $j\neq l$, are pairwise distinct. Therefore, by evaluating the above equality \eqref{eq:computation-1} at the Brauer classes $[A_j \otimes A_l]$, with $j\neq l$, we conclude that $\lambda_{j,l}\beta_j \beta_l =0$. Since $\beta_j, \beta_l >0$, this hence implies that $\lambda_{j,l}=0$ whenever $j\neq l$. Moreover, by evaluating the above equality \eqref{eq:computation-1} at the Brauer classes $[A_j]$, with $j \in \{i_1, \ldots, i_m\}$, we obtain the equality 
$$\sum_{j, l \in \{i_1, \ldots, i_m\}\,\text{with}\,\mathrm{sub}(j) \leq \mathrm{sub}(l)} \lambda_{j,l} (2\alpha_j \beta_j + \beta_j \alpha_l)=0\,.$$
Consequently, using the fact that $\lambda_{j, l}=0$ when $j\neq l$ and that $\alpha_j, \beta_j >0$, we conclude furthermore that $\lambda_{j, j}=0$ for every $j \in \{i_1,\ldots, i_m\}$. This concludes the proof of item b).
\begin{remark}[Lefschetz type]
Consider the Lefschetz class $\mathbb{L}:=[\mathbb{A}^1]$ in the Grothendieck ring of algebraic varieties. Given a sequence of integers $1 \leq n_1 < \ldots < n_r \leq d$, the associated symplectic flag variety $\mathrm{Flag}(2d;n_1, \ldots, n_r)$ admits a Schubert cell decomposition. Consequently, the Grothendieck class $[\mathrm{Flag}(2d;n_1, \ldots, n_r)]$ is of {\em Lefschetz type}, i.e., it can be written as a polynomial with $\mathbb{N}$-coefficients in the variable $\mathbb{L}$. Note that all the other twisted symplectic flag varieties are {\em not} of Lefschetz type! This follows from the fact that in all these cases the central simple $k$-algebra $A$ is {\em not} $k$-split, from the above computation \eqref{eq:measure} and from the fact that $\mu_{\mathrm{nc}}(\mathbb{L})=[k]$ (consult \cite[\S4.2.1]{book}).
\end{remark}

\bigskip

\noindent
{\bf Acknowledgements.} The author is grateful to Charles De Clercq for useful conversations.

\end{document}

\end{proof}